\documentclass[10pt,a4paper]{amsart}
\usepackage{amsmath,amsthm,amssymb,amsfonts}

\usepackage[pdftex]{color}
\usepackage[bookmarks=true,hyperindex,pdftex,colorlinks,citecolor=blue,
linkcolor=blue,urlcolor=blue]{hyperref}
\usepackage[english]{babel}

\usepackage{graphicx}

\theoremstyle{plain}
\newtheorem{theorem}{Theorem}[section]
\newtheorem{cor}[theorem]{Corollary}

\newtheorem{lemma}[theorem]{Lemma}

\theoremstyle{definition}

\newtheorem{remark}[theorem]{Remark}
\newtheorem{definition}[theorem]{Definition}

\newcommand{\N}{\mathbb{N}}
\newcommand{\C}{\mathbb{C}}

\newcommand{\Lin}{\mathcal{L}}

\newcommand{\eps}{\varepsilon}

\DeclareMathOperator{\re}{Re}

\DeclareMathOperator{\NA}{NA}

\renewcommand{\leq}{\leqslant}
\renewcommand{\geq}{\geqslant}

\title{The Bishop-Phelps-Bollob\'as property on the space of $c_0$-sum}

\author[G.~Choi]{Geunsu Choi}
\address[G.~Choi]{Department of Mathematics, Institute for Industrial and Applied Mathematics, Chungbuk National University, Cheongju, Chungbuk 28644, Republic of Korea}
\email{\texttt{chlrmstn90@gmail.com}}

\author[S.~K.~Kim]{Sun Kwang Kim}
\address[S.~K.~Kim]{Department of Mathematics, Chungbuk National University, Cheongju, Chungbuk 28644, Republic of Korea}
\email{\texttt{skk@chungbuk.ac.kr}}

\thanks{The first and second authors were supported by the National Research Foundation of Korea(NRF) grant funded by the Korea government(MSIT) [NRF-2020R1C1C1A01012267]. }

\date{\today}
\subjclass[2010]{Primary 46B04; Secondary  46B07, 46B20}

\keywords{Banach space, Norm attaining operator, The Bishop-Phelps-Bollob\'as theorem, Uniform convexity, Micro-transitivity, Bilinear forms}

\begin{document}
	
\begin{abstract}
The main purpose of this paper is to study Bishop-Phelps-Bollob\'as type properties on $c_0$ sum of Banach spaces. Among other results, we show that the pair $(c_0(X),Y)$ has the Bishop-Phelps-Bollob\'as property (in short, BPBp) for operators whenever $X$ is uniformly convex and $Y$ is (complex)  uniformly convex. We also prove that the pair $(c_0(X),c_0(X))$ has the BPBp for bilinear forms whenever $X$ is both uniformly convex and uniformly smooth. These extend the previously known results that $(c_0,Y)$ has the BPBp for operators whenever $Y$ is uniformly convex and $(c_0,c_0)$ has the BPBp for bilinear forms. We also obtain some results on a local BPBp which is called $\mathbf{L}_{p,p}$ for both operators and bilinear forms.
\end{abstract}

\maketitle

\section{Introduction}
The celebrated Bishop-Phelps theorem \cite{BP} states that every functional on a Banach space can be approximated by norm attaining ones. This astonishing result inspired many authors, and the denseness of norm attaining functions became a fruitful area. In 1970, B. Bollob\'as \cite{Bol} strengthened it by discovering a quantitative version that every functional and its almost norming points can be approximated by a norm attaining functional and its norm attaining point. After that Acosta, Aron, Garc\'ia and Maestre \cite{AAGM} defined a new notion for a pair of Banach spaces, which is called the Bishop-Phelps-Bollob\'as property for operators, and provided many notable results. In the past few years, many researchers found additional conditions for pairs of spaces having such property, and this is the main purpose of the present paper. To discuss them in details, we begin with introducing basic terminologies to remind.

We write $X$ and $Y$ for Banach spaces (over the scalar field $\mathbb{K}=\mathbb{R}$ or $\mathbb{C}$), and the notions $B_X$ and $S_X$ denote respectively the unit ball and unit sphere of $X$. We write by $X^*$ for the topological dual space of $X$. The space of all bounded linear operators from $X$ into $Y$ equipped with the operator norm will be denoted by $\Lin(X,Y)$, and the range space is omitted when $X=Y$. An operator $T \in \Lin(X,Y)$ is said to \emph{attain its norm} at $x_0 \in S_X$ if $\|Tx_0\|=\|T\|$, and in this case we write $T \in \NA(X,Y)$. We now introduce the main definition of our topic that we already mentioned.

\begin{definition}\cite[Definition 1.1]{AAGM}
A pair of Banach spaces $(X,Y)$ is said to have the \emph{Bishop-Phelps-Bollob\'as property} (in short, BPBp) \emph{for operators} whenever for each $\eps>0$ there exists $\eta(\eps)>0$ such that if $T \in S_{\Lin(X,Y)}$ and $x_0 \in S_X$ satisfy $\|Tx_0\|>1-\eta(\eps)$, then there exist $S \in S_{\Lin(X,Y)}$ and $z_0 \in S_{X}$ with
$$
\|Sz_0\|=1, \quad \|S-T\|<\eps \quad \text{and} \quad \|z_0-x_0\|<\eps.
$$
\end{definition}

In \cite{AAGM}, the authors proved that the pair $(X,Y)$ has the BPBp for operators when $X$ and $Y$ are both finite-dimensional, when $Y$ has property $\beta$, or when $X=\ell_1$ and $Y$ has so-called approximate hyperplane series property. Especially, they also proved that $(\ell_\infty^n,Y)$ has the BPBp for operators when $Y$ is \emph{uniformly convex}. The uniform convexity is a condition for a space $Y$ which states that for every $\eps>0$ there exists $\delta(\eps)>0$ such that $1 - \| \frac{x+y}{2} \| \geq \delta(\eps)$ whenever $\|x-y\| \geq \eps$ and $x,y \in B_Y$. In this case, we denote the infimum value taken among such $\delta(\eps)$ by the \emph{modulus of convexity} $\delta_Y(\eps)$ of $Y$. Kim \cite{K} improved this result once more by showing that $(c_0,Y)$ has the BPBp for operators when $Y$ is uniformly convex. While some conditions may get loosened on here, as Acosta \cite{A} showed that $(C_0(K),Y)$ has the BPBp for operators when $K$ is a locally compact Hausdorff space and $Y$ is $\C$-uniformly convex. A Banach space $Y$ is said to be \emph{$\C$-uniformly convex} if the \emph{modulus of $\C$-convexity}
$$
\delta_\C(\eps) := \inf_{x,y \in S_Y} \left\{ \sup_{|\lambda| = 1} \|x + \lambda \eps y\|-1 \right\}
$$
is strictly positive whenever $\eps>0$. Usually, the $\C$-uniform convexity is defined for complex Banach spaces and the uniform convexity implies the $\C$-uniform convexity. It is notable that $\C$-uniform convexity can also be defined for real Banach spaces, but this coincides with the uniform convexity. It is well known that every non-trivial complex $L_1(\mu)$ is $\C$-uniformly convex but the real one is not uniformly convex \cite{G}.  There are many other remarkable results on the BPBp for operators, which we refer to \cite{A-survey} for those who are interested.

Recently Dantas, Kim, Lee and Mazzitelli \cite{DKLM} introduced new versions of the BPBp, which they call the local BPBp for operators. We introduce here a specific type as follows.

\begin{definition}\cite[Definition 2.1]{DKLM}
A pair of Banach spaces $(X,Y)$ is said to have the $\mathbf{L}_{p,p}$ whenever for each $\eps>0$ and $x_0 \in S_X$ there exists $\eta(\eps,x_0)>0$ such that if $T \in S_{\Lin(X,Y)}$ satisfies $\|Tx_0\|>1-\eta(\eps,x_0)$, then there exists $S \in S_{\Lin(X,Y)}$ with
$$
\|Sx_0\|=1 \quad \text{and} \quad \|S-T\|<\eps.
$$
\end{definition}

Although it is not known whether the $\mathbf{L}_{p,p}$ implies the BPBp for operators or not, still it is a very intriguing convention as the $\mathbf{L}_{p,p}$ characterizes the strong subdifferentiability of a Banach space $X$ (see \cite[Theorem 2.3]{DKLM}). It is worth to remark that $c_0$ is one of typical example of infinite dimensional Banach space with a strong subdifferentiable norm, and it is shown in \cite{DKLM} that $(c_0,Y)$ has the $\mathbf{L}_{p,p}$ when $Y$ is uniformly convex.

These projects have been extended to the analogous type of properties for bilinear forms. One may define that a bilinear form $T \in \Lin^2(X \times Y)$ \emph{attains its norm} at $(x_0,y_0) \in S_X$ if $|T(x_0,y_0)|=\|T\|$, and the BPBp for bilinear form is defined as follows.
 \begin{definition}\cite{CS, DKLM2} A pair of Banach spaces $(X,Y)$ is said to have
\begin{enumerate}
\item[\textup{(a)}] the \emph{Bishop-Phelps-Bollob\'as property} (in short, BPBp) \emph{for bilinear forms} whenever for each $\eps>0$ there exists $\eta(\eps)>0$ such that if $T \in S_{\Lin^2(X \times Y)}$ and $(x_0,y_0) \in S_{X \times Y}$ satisfy that $|T(x_0,y_0)|>1-\eta$, then there exist $S \in S_{\Lin^2(X \times Y)}$ and $(u_0,v_0) \in S_{X \times Y}$ with
$$
|S(u_0,v_0)|=1, \quad \|u_0-x_0\|<\eps, \quad \|v_0-y_0\|<\eps \quad \text{and} \quad \|S-T\|<\eps.
$$
\item[\textup{(b)}] the $\mathbf{L}_{p,p}$ \emph{for bilinear forms} whenever for each $\eps>0$ and $(x_0,y_0) \in S_{X \times Y}$ there exists $\eta(\eps,x_0,y_0)>0$ such that if $T \in S_{\Lin^2(X \times Y)}$ satisfies $|T(x_0,y_0)|>1-\eta$, then there exists $S \in S_{\Lin^2(X \times Y)}$ with
$$
|S(x_0,y_0)|=1, \quad \text{and} \quad \|S-T\|<\eps.
$$
\end{enumerate}
\end{definition}
It is possible to equalize the space of bilinear forms $\Lin^2(X \times Y)$ and the space of bounded linear operators $\Lin(X, Y^*)$ (resp. $\Lin(Y,X^*)$) via an isometrically isomorphic correspondence between $T \in \Lin^2(X \times Y)$ and $L_T \in \Lin(X,Y^*)$ by $T(x,y) = (L_Tx)(y)$ (resp. $R_T \in \Lin(Y,X^*)$ by $T(x,y) = (R_Ty)(x)$). This clearly shows that the BPBp for bilinear forms on $(X,Y)$ implies the BPBp for operators on the pair $(X,Y^*)$. However, it is interesting that the the converse is not true in general \cite{CS}. We refer to \cite{ABGM} for more information on the pairs satisfying or failing the BPBp for bilinear forms, and \cite{DKLM2} for more information of the $\mathbf{L}_{p,p}$ for bilinear forms.

The main object we concern in this article is $c_0(X)= \bigl[ \oplus_{i=1}^\infty X \bigr]_{c_0}$, the $c_0$-sum of a Banach space $X$. It is shown in \cite{ACKLM} that the BPBp for operators is stable under $c_0$-sum of range spaces, but it is not known whether the reciprocal result on the domain space remains true. We note here just to be safe that some straightforward stability results can be obtained such as the pair $(c_0(X),C(K))$ when $X$ is Asplund and $K$ is a compact Hausdorff space due to \cite[Theorem 3.6]{CGK} as $c_0(X)$ is also Asplund (see \cite[p.213]{DU}, for instance). In this paper, we are interested in the operators from a Banach space of $c_0$-sum into a $\C$-uniformly convex space and the bilinear forms on the pair of Banach spaces of $c_0$-sums. In Section \ref{section:operator}, we first show that when $X$ is uniformly convex and $Y$ is  $\C$-uniformly convex, the pair $(c_0(X),Y)$ has the BPBp for operators which extends the result in \cite{K}. We also show that the pair $(c_0(X),Y)$ has the $\mathbf{L}_{p,p}$ when $X$ is micro-transitive and $Y$ is  $\C$-uniformly convex, extending \cite[Theorem 2.12]{DKLM}. In Section \ref{section:bilinear-form}, we move on to bilinear forms and prove that the pair $(c_0(X),c_0(X))$ has the BPBp for bilinear forms when $X$ is both uniformly convex and uniformly smooth. This is a generalization of the result in \cite{KLM} that $(c_0,c_0)$ has the BPBp for bilinear forms. Similarly to the aforementioned result, we find that the pair $(c_0(X),c_0(X))$ has the $\mathbf{L}_{p,p}$ for bilinear forms when $X$ is micro-transitive.

\section{The Bishop-Phelps-Bollob\'as property for operators on the $c_0$-sum}\label{section:operator}

In this section, we find new pairs $(c_0(X),Y)$ having the BPBp for operators. All the statements are given for Banach spaces based on $\C$ even though the real versions also hold. We omit the proofs of them since they can be proved with same arguments which are easier and simpler than the complex cases.  Before presenting the results, we need to introduce preliminary lemmas which will be used frequently throughout the article. In the following lemma, a convex series $(\alpha_i)_{i \in \N}\subset B_\mathbb{R}$ represents a sequence of non-negative numbers whose sum is less than or equal to $1$.

\begin{lemma}\cite[Lemma 3.3]{AAGM}\label{lemma:convex-series-estimate}
Let a number $0<\eta<1$ and a sequence $(z_i)_{i \in \N} \subset B_\mathbb{K}$ be given. If a convex series $(\alpha_i)_{i \in \N}\subset B_\mathbb{R}$ satisfies
$$
\re \sum_{i \in \N} \alpha_i z_i > 1- \eta,
$$
then we have for every $0<\eta'<1$ that
$$
\sum_{i \in A} \alpha_i > 1- \frac{\eta}{\eta'},
$$
where $A = \{ i \in \N \colon \re z_i > 1-\eta' \}$.
\end{lemma}

Recall that $\ell_\infty^A(X)$ ($A\subset\mathbb{N}$) denotes the space $\bigl[ \oplus_{i \in A} X \bigr]_{\ell_\infty}$, and $\ell_\infty^n(X):=\ell_\infty^{\{1,..,n\}}(X)$. 

\begin{lemma}\cite[Lemma 2.3]{A}\label{lemma:uniformly-convex-projection}
Let $X$ be a Banach space and $Y$ be a $\C$-uniformly convex Banach space with modulus of $\C$-convexity $\delta_\C$. For a fixed $\eps>0$, if $T \in B_{\Lin(c_0(X),Y)}$ and $A \subset \N$ satisfy that $\|TP_A\| \geq 1- \frac{\delta_\C(\eps)}{1+\delta_\C(\eps)}$, then $\|T(I-P_A)\| \leq \eps$ where $P_A \colon c_0(X) \to \ell_\infty^A(X) \subset c_0(X)$ is a projection on the components in $A$. Analogously, if $T \in S_{\Lin(\ell_\infty^n(X),Y)}$ and $A \subset \{1, \ldots, n \}$ satisfy that $\|TP_A\| \geq 1- \frac{\delta_\C(\eps)}{1+\delta_\C(\eps)}$, then $\|T(I-P_A)\| \leq \eps$.
\end{lemma}

The statements are different from the original ones, but we omit the proofs since they are just slight modifications of them. Especially, we use Lemma \ref{lemma:convex-series-estimate} in order to get the following estimation which will be frequently used.

\begin{lemma}\label{lemma:convex-series-estimate2} Let $X$ and $Y$ be Banach spaces and $0<\eta<1$ be given. Assume that $T \in S_{\Lin(c_0(X),Y)}$, $y_0^* \in S_{Y^*}$ and $x_0 \in S_{c_0(X)}$ satisfy that
$$y_0^*(Tx_0) = \|Tx_0\| > 1-\eta.$$
Then, for $0<\eta'<1$, we have
$$
 \sum_{i \in A} \|(T^*y_0^*)(i)\| > 1- \frac{\eta}{\eta'}.
$$
  where $A := \left\{i \in N \colon \re \left[(T^*y_0^*)(i) \right] (x_0(i)) > \left(1-\eta'\right)\|(T^*y_0^*)(i)\|\right\}$.

In particular, 
$$\re \sum_{i\in A}\left[(T^*y_0^*)(i)\right] (x_0(i))>\left(1- \frac{\eta}{\eta'}\right)\left(1-\eta'\right).$$
\end{lemma}

\begin{proof}
We define $N\subset \N$ by
$$
N:=\{i \in \N \colon \|(T^*y_0^*)(i)\|\neq 0\}.
$$
Then, it is clear that 
$$A = \left\{i \in N \colon \re \left[ \frac{(T^*y_0^*)(i)}{\|(T^*y_0^*)(i)\|} \right] (x_0(i)) > 1-\eta'\right\}.
$$
Since 
$$
y_0^*(Tx_0)= \sum_{i=1}^\infty \left[ (T^*y_0^*)(i) \right] (x_0(i)) = \sum_{i\in N} \|(T^*y_0^*)(i)\|\left[ \frac{(T^*y_0^*)(i)}{\|(T^*y_0^*)(i)\|} \right] (x_0(i)),
$$
by Lemma \ref{lemma:convex-series-estimate}, we have that
$$
 \sum_{i \in A} \|(T^*y_0^*)(i)\| > 1- \frac{\eta}{\eta'}.
$$
\end{proof}

We are now ready to present the main theorem of this section.

\begin{theorem}\label{theorem:c_0-uniformly-convex-BPBp}
Let $X$ be a uniformly convex Banach space and $Y$ be a $\C$-uniformly convex Banach space. Then, the pair $(c_0(X),Y)$ has the BPBp for operators.
\end{theorem}

\begin{proof}
Let $0<\eps<1$ be given. Set $\eta(\eps) := \min \left\{ \frac{\eps}{16}, \frac{\delta_{\mathbb{C}}(\frac{\eps}{16})}{1+\delta_{\mathbb{C}}(\frac{\eps}{16})}, \delta_X(\frac{\eps}{2})\right\}$ where $\delta_X$ is the modulus of convexity of $X$ and $\delta_{\mathbb{C}}$ is the modulus of $\C$-convexity of $Y$. Assume that $T \in S_{\Lin(c_0(X),Y)}$ and $x_0 \in S_{c_0(X)}$ satisfy that
$$
\|Tx_0\| > 1-\frac{\eta^6}{64}.
$$
Choose $y_0^* \in S_{Y^*}$ so that $y_0^*(Tx_0) = \|Tx_0\| > 1-\frac{\eta^6}{64}$ and define subsets $N,A \subset \N$ respectively by
$$
 A := \left\{i \in N \colon \re \left[(T^*y_0^*)(i) \right] (x_0(i)) > \left(1-\frac{\eta^3}{8}\right)\|(T^*y_0^*)(i)\|\right\}.
$$

By Lemma \ref{lemma:convex-series-estimate2}, we have that
$$
 \sum_{i \in A} \|(T^*y_0^*)(i)\| > 1- \frac{\eta^3}{8}.
$$
and so Lemma \ref{lemma:uniformly-convex-projection} shows that
$$
\|TP_A-T\| < \frac{\eps}{16}.
$$
The canoncial restriction $\hat{T} \in B_{\Lin(\ell_\infty^A(X),Y)}$ of $T$ and an element $\hat{x}_0 = (\hat{x}_0(i))_{i \in A} = \left(\frac{x_0(i)}{\|x_0(i)\|}\right)_{i \in A} \in S_{\ell_\infty^A(X)}$ satisfy that
$$
\bigl\|\hat{T}\hat{x}_0 \bigr\| \geq \re \sum_{i \in A} \|(T^*y_0^*)(i)\|\left[ \frac{(T^*y_0^*)(i)}{\|(T^*y_0^*)(i)\|} \right] (\hat{x}_0(i))  >1-\frac{\eta^3}{4}\quad \text{and} \quad \|\hat{x}_0(i)-x_0(i)\| <\frac{\eta^3}{8} \quad \text{for } i \in A.
$$

Choose $y_1^* \in S_{Y^*}$ so that $\re y_1^*(\hat{T}\hat{x}_0) = \|\hat{T}\hat{x}_0\|$ and define $R \in \Lin(\ell_\infty^A(X),Y)$ by
$$
R(z) := \hat{T}z + \eta y_1^*(\hat{T}z) \frac{\hat{T}\hat{x}_0}{\bigl\| \hat{T}\hat{x}_0 \bigr\|}
$$
for $z \in \ell_\infty^A(X)$. As $\ell_\infty^A(X)$ has the RNP, by Bourgain \cite{Bou} there exists $Q \in \NA(\ell_\infty^A(X),Y)$ such that $Q$ attains its norm at $w_0 \in S_{\ell_\infty^A(X)}$, $\|Q\|=\|R\|$ and $\|Q-R\|<\frac{\eta^3}{4}$. We here deduce that
$$
1-\frac{\eta^3}{4} + \eta\left(1-\frac{\eta^3}{4}\right) \leq \|R\hat{x}_0\| \leq \|R\|=\|Qw_0\| \leq \|Qw_0-Rw_0\|+\|Rw_0\| \leq \frac{\eta^3}{4} + 1 + \eta \bigl|y_1^*(\hat{T}w_0)\bigr|.
$$
If we rotate $w_0$ if it is necessary, we may assume $\bigl|y_1^*(\hat{T}w_0)\bigr| = \re y_1^*(\hat{T}w_0)$, and we obtain from above that
$$
\re y_1^*(\hat{T}w_0) \geq 1 - \frac{\eta^2}{2} - \frac{\eta^3}{4} \geq 1-\eta^2.
$$
It follows that
$$
\re y_1^*\left( \hat{T} \left( \frac{w_0+\hat{x}_0}{2} \right) \right) \geq 1 - \frac{\eta^2 + \frac{\eta^3}{4}}{2} \geq 1- \eta^2.
$$
Define a subset $B \subset A$ by
$$
B := \left\{ i \in A \colon \re \left[\hat{T}^*y_1^*(i)\right]  \left( \frac{w_0 + \hat{x}_0}{2} (i) \right)  > \left(1 - \eta \right) \|\hat{T}^*y_1^*(i)\| \right\}.
$$
We use the  Lemma \ref{lemma:convex-series-estimate2} again to deduce
$$
\|\hat{T}P_B\| \geq \sum_{i \in B} \|\hat{T}^*y_1^*(i)\| > 1- \eta.
$$
As $Y$ is $\C$-uniformly convex, we get by Lemma \ref{lemma:uniformly-convex-projection} that
$$
\|\hat{T}(I-P_B)\| < \frac{\eps}{16}.
$$
 Moreover, for any $i \in B$, we see that
$$
\left\| \frac{w_0+\hat{x}_0}{2}(i)\right\| > 1-\eta,
$$
and so by uniform convexity of $X$ we can derive that
$$
\|w_0(i) - \hat{x}_0(i)\| < \frac{\eps}{2}.
$$

Now, define $\widetilde{S} \in \Lin(\ell_\infty^A(X),Y)$ by $\widetilde{S} := QP_B + Q(I-P_B)U$ where $U \in B_{\Lin(\ell_\infty^A(X))}$ is chosen so that $U\bigl(E_i(\hat{x_0}(i))\bigr) = E_i(w_0(i))$ for every $i \in A$ and $E_i \colon X \to \ell_\infty^A(X)$ is the $i^{\text{th}}$ injection map. Note that $\|\widetilde{S}\|\leq \|Q\|$. Let $S$ be the canonical extension of $\frac{\widetilde{S}}{\|\widetilde{S}\|}$ and define
\begin{displaymath}
z_0(i) :=\left\{\begin{array}{@{}cl}
\displaystyle \, w_0(i) & \text{if } i \in B \\
\displaystyle \, \hat{x}_0(i) & \text{if } i \in A \setminus B \\
\displaystyle \, x_0(i) & \text{otherwise}.
\end{array} \right.
\end{displaymath}
It is clear that  $\|Sz_0\| =\frac{\|Qw_0\|}{\|\widetilde{Q}\|}=1$.
Also,
$$
\|z_0-x_0\| \leq \max\left\{\sup_{i \in B} (\|w_0(i) - \hat{x}_0(i)\| + \|\hat{x}_0(i)-x_0(i)\|),~\sup_{i \in A} \|\hat{x}_0(i)-x_0(i)\|\right\} < \frac{\eps}{2} + \frac{\eta^3}{8} < \eps.
$$
Finally, we have
\begin{align*}
\|S-T\| 
&\leq\|SP_A-TP_A\|+\|TP_A-T\|\\
&=\left\|\frac{\widetilde{S}}{\|\widetilde{S}\|}-\hat{T}\right\|+\|TP_A-T\|\\
&\leq \left\|\frac{\widetilde{S}}{\|\widetilde{S}\|}-\widetilde{S}\right\|+\|\widetilde{S} - Q\| + \|Q-R\|+\|R-\hat{T}\|+\|TP_A-T\| \\
&< \left|1- \|\widetilde{S}\|\right|+\|QP_B + Q(I-P_B)U - Q \| +\frac{\eta^3}{4}+ \eta+\frac{\eps}{16}  \\
&\leq  \left|1- \|R\|\right|+2\|Q(I-P_B)\|+\frac{\eta^3}{4}+ \eta+\frac{\eps}{16}  \\
&<  \frac{\eps}{16} + \eta +2\|\hat{T}(I-P_B)\|+2\|\hat{T}-Q\|+\frac{\eta^3}{4}+ \eta+\frac{\eps}{16}  \\
&<\frac{\eps}{16} + \eta +\frac{\eps}{8}+3\left(\frac{\eta^3}{4}+ \eta\right)+\frac{\eps}{16}   < \eps,
\end{align*}
which finishes the proof.
\end{proof}

\begin{remark}
The same proof of Theorem \ref{theorem:c_0-uniformly-convex-BPBp} with the same constant holds for the case of finite sum of $X$ with $\ell_\infty$ norm, and this fact will be applied to prove Theorem \ref{theorem:c_0-L-p,p}. 
\end{remark}
As an immediate consequence of Theorem \ref{theorem:c_0-uniformly-convex-BPBp} and \cite[Proposition 2.4]{ACKLM}, we can deduce that the BPBp for endomorphisms on $c_0(X)$ holds for every uniformly convex $X$.

\begin{cor}
Let $X$ be a uniformly convex Banach space. Then, $(c_0(X), c_0(X))$ has the \textup{BPBp} for operators.
\end{cor}

We devote the rest of this section to study the local BPBp for operators on $c_0$-sum of Banach spaces. As we commented in the introduction, the pair $(c_0,Y)$ has the $\mathbf{L}_{p,p}$ when $Y$ is $\C$-uniformly convex (see \cite[Theorem 2.12]{DKLM}). To extend this result, we recall one notion on a space. A Banach space $X$ is said to be \emph{micro-transitive} if for every $\eps>0$ there exists $0<\theta(\eps)<\eps$ so that whenever $x,y \in S_X$ satisfy $\|x-y\|<\theta(\eps)$, there exists an isometry $U \colon X \to X$ such that
$$
Ux = y \quad \text{and} \quad \|U-I\|<\eps,
$$
where $I \colon X \to X$ denotes the canonical identity operator. Hilbert spaces such as Euclidean spaces and $\ell_2$ are known to be micro-transitive, see \cite{AMS}. Note that every micro-transitive Banach space is uniformly convex and uniformly smooth, but the converse does not hold in general (see \cite{CDKKLM}).

\begin{theorem}\label{theorem:c_0-L-p,p}
Let $X$ be a micro-transitive Banach space and $Y$ be a $\C$-uniformly convex Banach space. Then, the pair $(c_0(X),Y)$ has the $\mathbf{L}_{p,p}$.
\end{theorem}

\begin{proof}
From Theorem \ref{theorem:c_0-uniformly-convex-BPBp} and the fact that micro-transitivity implies the uniform convexity, the pair $(c_0(X),Y)$ has the BPBp for operators with the function $\eta(\cdot)$ given in the definition of the BPBp. Let $\theta(\cdot)$ be the function in the definition of micro-transitivity of $X$ and $\delta_{\mathbb{C}}$ be the modulus of $\C$-convexity of $Y$.
 
For given $0<\eps<1$ and $x_0 \in S_{c_0(X)}$, set $\gamma=\gamma(\eps) := \min \left\{ m_{x_0}, \eta\left(\frac{\theta(\frac{\eps}{3})}{2}\right), \frac{\delta_{\mathbb{C}}\left(\frac{\eps}{6}\right)}{1+\delta_{\mathbb{C}}\left(\frac{\eps}{6}\right)} \right\}$ where $m_{x_0} := \min_{i \in \N} \{ 1- \|x_0(i)\| \colon \|x_0(i)\|<1 \}$ and assume that $T \in S_{\Lin(c_0(X),Y)}$ and $x_0 \in S_{c_0(X)}$ satisfy that
$$
\|Tx_0\| > 1-\frac{\gamma^2}{4}.
$$
Choose $y_0^* \in S_{Y^*}$ so that
$$
\sum_{i=1}^\infty \bigl[ (T^*y_0^*)(i) \bigr] (x_0(i)) = y_0^*(Tx_0) = \|Tx_0\| > 1-\frac{\gamma^2}{4}.
$$
Similarly to the proof of Theorem \ref{theorem:c_0-uniformly-convex-BPBp}, define a subset $A\subset N$ by
$$
A := \left\{i \in N \colon \re \left[ {(T^*y_0^*)(i)} \right] (x_0(i)) > \left(1-\frac{\gamma}{2}\right){\|(T^*y_0^*)(i)\|}\right\}.
$$
Then, from the definition of $\gamma$ we have that $\|x_0(i)\|=1$ for each $i\in A$ and so $A$ is finite. By Lemma \ref{lemma:convex-series-estimate2} and Lemma \ref{lemma:uniformly-convex-projection}, we deduce that
$$
 \sum_{i \in A} \|(T^*y_0^*)(i)\| > 1- \frac{\gamma}{2}
\text{~and~}
\|TP_A-T\| < \frac{\eps}{6}.
$$
Hence, the canoncial restriction $\hat{T} \in B_{\Lin(\ell_\infty^A(X),Y)}$ of $T$ and $\hat{x}_0 = (\hat{x}_0(i))_{i \in A} = \left(x_0(i)\right)_{i \in A} \in S_{\ell_\infty^A(X)}$ satisfy that
$$
\bigl\|\hat{T}\hat{x}_0 \bigr\| \geq \re \sum_{i \in A} \|(T^*y_0^*)(i)\|\left[ \frac{(T^*y_0^*)(i)}{\|(T^*y_0^*)(i)\|} \right] (\hat{x}_0(i))  >1-{\gamma},
$$
and so it holds that $\bigl\|\widetilde{T}\hat{x}_0 \bigr\| > 1- \gamma$
where $\widetilde{T}  := \frac{\hat{T}}{\|\hat{T}\|}\in S_{\Lin(\ell_\infty^A(X),Y)}$.

We now apply the version for the finite $\ell_\infty$-sum of Theorem \ref{theorem:c_0-uniformly-convex-BPBp} to obtain a new operator $\hat{S} \in S_{\Lin(\ell_\infty^A(X)),Y)}$ and a point $z_0 \in S_{\ell_\infty^A(X)}$ satisfying
$$
\|\hat{S}z_0\|=1, \quad \|z_0-\hat{x}_0\|<\frac{\theta(\frac{\eps}{3})}{2} \quad \text{and} \quad \|\hat{S}-\widetilde{T}\|<\frac{\theta(\frac{\eps}{3})}{2}.
$$
If we consider $\hat{z}_0 = (\hat{z}_0(i))_{i \in A} = \left( \frac{z_0(i)}{\|z_0(i)\|} \right)_{i \in A} \in S_{\ell_\infty^A(X)}$ we get $\|\hat{S}\hat{z}_0\|=1$ from the convexity of the norm since 
$z_0=\frac{z_0+\left(\hat{z}_0-z_0\right)+z_0-\left(\hat{z}_0-z_0\right)}{2}$ and $\left\|z_0\pm\left(\hat{z}_0-z_0\right)\right\|\leq 1$. Moreover, we have that  
$$
\|\hat{z}_0-\hat{x}_0\| \leq \|\hat{z}_0 - z_0\| + \|z_0 - \hat{x}_0\|=\max_{i\in A}(1-\|z_0(i)\|)+ \|z_0 - \hat{x}_0\|< \theta\left(\frac{\eps}{3}\right).
$$
From the micro-transitivity of $X$, we take an isometry $U_i \colon X \to X$ for each $i \in \N$ so that $\|U_i-I\|<\frac{\eps}{3}$, $U_ix_0(i)=\hat{z}_0(i)$ for $i\in A$ and $U_i=I$ for $i \notin A$.

Now, define $\widetilde{S} \colon c_0(X) \to Y$  be the canonical extension of $\hat{S}$. Then, we find the desired operator $S \in S_{\Lin(c_0(X),Y)}$ by $Sx := \widetilde{S}\left((U_iP_ix)_i\right)$ for each $x\in c_0(X)$ where $P_i \colon c_0(X) \to X$ is the coordinate projection. Indeed, it is clear that $\|Sx_0\|=\|\hat{S}\hat{z}_0\|=1$, and furthermore, we have that
\begin{align*}
\|S-T\|
&\leq \|S-\widetilde{S}\|+\|\widetilde{S}-TP_A\|+\|TP_A-T\|\\
 &\leq \|S - \widetilde{S}\| + \|\hat{S} - \widetilde{T}\| + \|\widetilde{T} - \hat{T}\| + \|TP_A - T \| \\
&< \|\widetilde{S}\|\max_{i\in A} \|U_i-I\|+\frac{\theta(\frac{\eps}{3})}{2} + \left|1-\|\hat{T}\|\right| + \frac{\eps}{6} \\
&< \frac{\eps}{3}+\frac{\theta(\frac{\eps}{3})}{2} + \frac{\eps}{6} + \frac{\eps}{6} < \eps,
\end{align*}
as desired.
\end{proof}

\section{The Bishop-Phelps-Bollob\'as property for bilinear forms on the $c_0$-sum}\label{section:bilinear-form}

In this section, we discuss the Bishop-Phelps-Bollob\'as property for bilinear forms. All the scalar field is complex since it is crucial that the complex space $\ell_1$-sum ($\ell_1(X)$) of a $\mathbb{C}$-uniformly convex space $X$ is also $\mathbb{C}$-uniformly convex \cite{DT}. To do so, we first need a lemma which acts simlarly as Lemma \ref{lemma:uniformly-convex-projection} does, but it is based on bilinear forms. In the following lemma, the notion $P_{A_1,A_2} \colon c_0(X) \times c_0(X) \to \ell_\infty^{A_1}(X) \times \ell_\infty^{A_2}(X)\subset c_0(X) \times c_0(X)$ denotes the canonical projection.

\begin{lemma}\label{lemma:bilinear-projection}
Let $X$ be a uniformly convex and uniformly smooth Banach space and let $T \in B_{\Lin^2(c_0(X) \times c_0(X))}$ be given. For every $\eps>0$, there exists $0<\gamma(\eps)<1$ such that if finite subsets $A_L,A_R\subset \mathbb{N}$ satisfy that
$$
\|TP_{A_L,A_R}\|> 1- \gamma(\eps),
$$
then 
$$
\|T-TP_{A_L,A_R}\|<\eps.
$$
\end{lemma}

\begin{proof}
Let $\delta_\C$ be the modulus of $\C$-convexity of $\ell_1(X^*)$. For a fixed $0<\eps<1$, put $\gamma(\eps) := \frac{\delta_\C(\frac{\eps}{2})}{1+\delta_\C(\frac{\eps}{2})}$ and assume that $\|TP_{A_L,A_R}\|> 1- \gamma(\eps)$. Since $\Lin^2(c_0(X) \times c_0(X))$ is isometrically isomorphic to $\Lin(c_0(X), \ell_1(X^*))$  and 
$$\|TP_{A_L,\mathbb{N}}\|\geq \|TP_{A_L,A_R}\|> 1- \gamma(\eps),$$
Lemma \ref{lemma:uniformly-convex-projection} gives that $\|T-TP_{A_L,\mathbb{N}}\|<\frac{\eps}{2}$. We apply the same argument to obtain that $\|TP_{A_L,\mathbb{N}}-TP_{A_L,A_R}\|<\frac{\eps}{2}$ from the equality $P_{A_L,A_R}=P_{A_L,\mathbb{N}}P_{\mathbb{N},A_R}$. Finally, we have that
$$
\|T - TP_{A_L,A_R}\| \leq \|T - TP_{A_L,\mathbb{N}}\| + \|TP_{A_L,\mathbb{N}} - TP_{A_L,\mathbb{N}}P_{\mathbb{N},A_R}\| < \frac{\eps}{2} + \frac{\eps}{2} = \eps.
$$

\end{proof}

\begin{theorem}\label{theorem:c_0-bilinear-form-BPBp}
Let $X$ be a uniformly convex and uniformly smooth Banach space. Then, the pair $(c_0(X),c_0(X))$ has the BPBp for bilinear forms.
\end{theorem}

\begin{proof}
Let $0<\eps<1$ be given. Set $\eta(\eps):= \min \left\{\frac{\eps}{2^4}, \gamma\left(\frac{\eps}{2^4}\right), \delta_X\left(\frac{\eps}{2}\right)\right\}$ where $\gamma$ is the function given in Lemma \ref{lemma:bilinear-projection} and $\delta_X$ is the modulus of convexity of $X$. Let $T \in \Lin^2(c_0(X) \times c_0(X))$ with $\|T\|=1$ and $(x_L,x_R) \in S_{c_0(X)} \times S_{c_0(X)}$ satsify that
$$
|T(x_L,x_R)| > 1-\frac{\eta^{12}}{2^{22}}.
$$
By taking a suitable rotation, we may assume that 
$$
\re T(x_L,x_R) > 1-\frac{\eta^{12}}{2^{22}},
$$
and define a subset $A_R$ of $\N$ by
$$A_R = \left\{ i \in \mathbb{N} \colon \re \left[{(L_Tx_L)(i)}\right] (x_R(i)) > \left(1- \frac{\eta^6}{2^{11}}\right)\|(L_Tx_L)(i)\| \right\}.
$$
It is clear that $\|x_R(i)\|>1-\frac{\eta^6}{2^{11}}$ for $i \in A_R$, and by Lemma \ref{lemma:convex-series-estimate2} we have that
$$
\sum_{i \in A_R} \|(L_Tx_L)(i)\| > 1- \frac{\eta^6}{2^{11}}.
$$
We also deduce
\begin{align*}
\re T(x_L,\hat{x}_R) &= \re \sum_{i \in A_R} \|(L_T x_L)(i)\| \left[ \frac{(L_Tx_L)(i)}{\|(L_Tx_L)(i)\|} \right] (\hat{x}_R(i)) \\
&\geq \re \sum_{i \in A_R} \|(L_T x_L)(i)\| \left[ \frac{(L_Tx_L)(i)}{\|(L_Tx_L)(i)\|} \right] (x_R(i))  > 1-\frac{\eta^6}{2^{10}},
\end{align*}
where $\hat{x}_R$ is defined by $\hat{x}_R(i) = \frac{x_R(i)}{\|x_R(i)\|}$ for $i\in A_R$ and $0$ otherwise. 
 Again, we define a subset $A_L$ of $\mathbb{N}$ by
$$
 A_L = \left\{ i \in \mathbb{N} \colon \re \left[ {(R_T\hat{x}_R)(i)} \right] (x_L(i)) >\left(1- \frac{\eta^3}{2^5}\right)\|(R_T\hat{x}_R)(i)\| \right\}.
$$
By Lemma \ref{lemma:convex-series-estimate2}, we have that
$$
\sum_{i \in A_L} \|(R_T\hat{x}_R)(i)\| > 1- \frac{\eta^3}{2^5}
$$
and $\|x_L(i)\|>1-\frac{\eta^3}{2^5}$ for $i \in A_L$. Hence, elements $\hat{x}_R$ and $\hat{x}_L$ defined by $\hat{x}_L(i) = \frac{x_L(i)}{\|x_L(i)\|}$ for $i\in A_L$ and $0$ otherwise satisfy that
$$
\re T(\hat{x}_L,\hat{x}_R)= \re \sum_{i \in A_L} \|(R_T \hat{x}_R)(i)\| \left[ \frac{(R_T\hat{x}_R)(i)}{\|(R_T\hat{x}_R)(i)\|} \right] (\hat{x}_L(i)) > 1-\frac{\eta^3}{2^4}
$$ 
and that
$$
\|\hat{x}_L(i)-x_L(i)\| < \frac{\eta^3}{2^5} \quad \text{for } i \in A_L \quad \text{and} \quad \|\hat{x}_R(i)-x_R(i)\| < \frac{\eta^6}{2^{11}} \quad \text{for } i \in A_R.
$$

Moreover, from Lemma \ref{lemma:bilinear-projection}, we get
$$
\|TP_{A_L,A_R}-T\|<\frac{\eps}{2^4}
$$

From now on, we will consider $\hat{x}_L$ and $\hat{x}_R$ as elements in $\ell_\infty^{A_L}(X)$ and $\ell_\infty^{A_R}(X)$ respectively for the convenience. With the canonical restriction $\hat{T} \in B_{\Lin^2(\ell_\infty^{A_L}(X) \times \ell_\infty^{A_R}(X))}$ of $T$, define $R \in \Lin^2(\ell_\infty^{A_L}(X) \times \ell_\infty^{A_R}(X))$ by
$$
R(z_L,z_R):= \hat{T}(z_L,z_R) + \eta \hat{T}(z_L, \hat{x}_R) \hat{T}(\hat{x}_L,z_R) \frac{\bigl|\hat{T}(\hat{x}_L,\hat{x}_R)\bigr|}{\hat{T}(\hat{x}_L,\hat{x}_R)}
$$
for $z_L\in \ell_\infty^{A_L}(X)$ and $z_R \in \ell_\infty^{A_R}(X)$. 

Since $\ell_\infty^{A_L}(X)$ and $\ell_\infty^{A_R}(X)$ have the Radon-Nikod\'ym property, according to \cite[Theorem 1]{AFW} there exists $Q \in \Lin^2(\ell_\infty^{A_L}(X) \times \ell_\infty^{A_R}(X))$ with $\|Q\|=\|R\|$ such that $Q$ attains its norm at $(w_L,w_R) \in S_{\ell_\infty^{A_L}(X)} \times S_{\ell_\infty^{A_R}(X)}$ and $\|Q-R\|<\frac{\eta^3}{2^4}$.

 Note that
\begin{equation}\label{eq:lower-bound-of-R}
\|R\| \geq |R(\hat{x}_L,\hat{x}_R)| = \bigl|\hat{T}(\hat{x}_L,\hat{x}_R)\bigr| + \eta \bigl|\hat{T}(\hat{x}_L,\hat{x}_R)\bigr|^2> \left(1-\frac{\eta^3}{2^4}\right) + \eta\left(1-\frac{\eta^3}{2^4}\right)^2.
\end{equation}
On the other hand,
\begin{equation}\label{eq:upper-bound-of-R}
\|R\|=\|Q\|=|Q(w_L,w_R)|\leq \frac{\eta^3}{2^4} +|R(w_L,w_R)|\leq \frac{\eta^3}{2^4} + 1 + \eta \bigl|\hat{T}(w_L,\hat{x}_R)\bigr| \bigl|\hat{T}(\hat{x}_L,w_R)\bigr|.
\end{equation}
Combining \eqref{eq:lower-bound-of-R} with \eqref{eq:upper-bound-of-R}, we obtain that
$$
1 - \frac{\eta^2}{2^2} \leq \left(1-\frac{\eta^3}{2^4}\right)^2 - \frac{\eta^2}{2^3} \leq \min \left\{ \bigl|\hat{T}(w_L,\hat{x}_R)\bigr|, \bigl|\hat{T}(\hat{x}_L,w_R)\bigr| \right\}.
$$
Here, we may assume that $\bigl|\hat{T}(w_L,\hat{x}_R)\bigr| = \re \hat{T}(w_L,\hat{x}_R)$ and $\bigl|\hat{T}(\hat{x}_L,w_R)\bigr| = \re \hat{T}(\hat{x}_L,w_R)$ by rotating $w_L$ and $w_R$ if necessary. For a set 
$$
B_L := \left\{ i \in A_L \colon \re \left[ \left(R_{\hat{T}} \hat{x}_R \right)(i) \right] \left( \frac{w_L+\hat{x}_L}{2} (i) \right) > \left(1-\frac{\eta}{2}\right) \left\|  \left(R_{\hat{T}} \hat{x}_R \right)(i) \right\| \right\},
$$
Lemma \ref{lemma:convex-series-estimate2} and Lemma \ref{lemma:bilinear-projection} show that
$$\re \hat{T} \left( P_{B_L}\left(\frac{w_L + \hat{x}_L}{2}\right) , \hat{x}_R \right)>1-\eta \quad \text{and} \quad \|\hat{T}-\hat{T}P_{B_L,A_R}\|=\|\hat{T}P_{A_L\setminus B_L,A_R}\|<\frac{\eps}{2^4}.$$
Similarly, a set 
$$
B_R := \left\{ i \in A_R \colon \re \left[ \left(L_{\hat{T}} \hat{x}_L \right)(i) \right] \left( \frac{w_R+\hat{x}_R}{2} (i) \right) > \left(1-\frac{\eta}{2}\right) \left\|  \left(L_{\hat{T}} \hat{x}_L \right)(i) \right\| \right\}
$$
satisfies that 
 $$\re \hat{T} \left( \hat{x}_L ,  P_{B_R}\left(\frac{w_R + \hat{x}_R}{2} \right)\right)>1-\eta \quad \text{and} \quad \|\hat{T}- \hat{T}P_{A_L,B_R}\|=\|\hat{T}P_{A_L,A_R\setminus B_R}\|<\frac{\eps}{2^4}.
$$
Thus, we deduce that 
$$
\|\hat{T} - \hat{T}P_{B_L,B_R}\| \leq \|\hat{T} - \hat{T}P_{B_L,A_R}\|+\|\hat{T}P_{B_L,A_R}- \hat{T}P_{B_L,B_R}\|\leq \|\hat{T} - \hat{T}P_{B_L,A_R}\|+\|\hat{T}- \hat{T}P_{A_L,B_R}\|< \frac{\eps}{2^3}.
$$
Since we have for each $i\in B_L$ and $j\in B_R$ that
$$
\left\| \frac{w_L + \hat{x}_L}{2} (i) \right\| > 1-\eta \quad \text{and} \quad \left\| \frac{w_R + \hat{x}_R}{2} (j) \right\| > 1-\eta,
$$
 the uniform convexity of $X$ implies that 
$$
\|w_L(i) - \hat{x}_L(i)\| < \frac{\eps}{2} \quad \text{and} \quad \|w_R(i) - \hat{x}_R(j)\| < \frac{\eps}{2}.
$$

Let $U_i \in B_{\Lin(X)}$ and $V_i \in B_{\Lin(X)}$ be chosen so that $U_i(\hat{x}_L(i)) = w_L(i)$ for each $i\in A_L$ and $V_i(\hat{x}_R(i)) = w_R(i)$ for each $i \in A_R$. Now, define $\widetilde{S} \in \Lin^2(\ell_\infty^{A_R}(X) \times \ell_\infty^{A_R}(X))$  by
 $$\widetilde{S} := Q\left(P_{B_L,B_R}+ WP_{A_L\setminus B_L,A_R\setminus B_R}\right)$$
where $W \colon \ell_\infty^{A_L}(X) \times \ell_\infty^{A_R}(X) \to \ell_\infty^{A_L}(X) \times \ell_\infty^{A_R}(X)$ is given by $W(z_L,z_R) := \left((U_iz_L(i))_i, (V_iz_R(i))_i\right)$ for $z_L=(z_L(i))_{i \in A_L} \in \ell_\infty^{A_L}(X)$ and $z_R=(z_R(i))_{i \in A_R} \in \ell_\infty^{A_R}(X)$. 

Let $S\in \Lin^2(c_0(X) \times c_0(X))$ be the canonical extension of $\frac{\widetilde{S}}{\|\widetilde{S}\|}$ and define
\begin{displaymath}
u_L(i) :=\left\{\begin{array}{@{}cl}
\displaystyle \, w_L(i) & \text{if } i \in B_L \\
\displaystyle \, \hat{x}_L(i) & \text{if } i \in A_L \setminus B_L \\
\displaystyle \, x_L(i) & \text{otherwise},
\end{array} \right.
\end{displaymath}
\begin{displaymath}
u_R(j) :=\left\{\begin{array}{@{}cl}
\displaystyle \, w_R(j) & \text{if } j \in B_R \\
\displaystyle \, \hat{x}_R(j) & \text{if } j \in A_R \setminus B_R \\
\displaystyle \, x_R(j) & \text{otherwise}.
\end{array} \right.
\end{displaymath}

It remains to prove that those $S$ and $(u_L,u_R)$ are the desired bilinear form and its norm attaining point. First, we have that 
$$
|S(u_L,u_R)|=\frac{|\widetilde{S}(u_L,u_R)|}{\|\widetilde{S}\|}=\frac{|Q(w_L,w_R)|}{\|Q\|}=1
$$
and that 
\begin{align*}
\|(u_L,u_R)-(x_L,x_R)\| &\leq \max_{i \in A_L, j \in A_R} \Bigl(\|(u_L(i),u_R(j)) - (\hat{x}_L(i),\hat{x}_R(j)) \| + \| (\hat{x}_L(i), \hat{x}_R(j)) - (x_L(i),x_R(j)) \|\Bigr) \\
&\leq \max \left\{ \sup_{i \in B_L} \|w_L(i) - \hat{x}_L(i)\|, \sup_{j \in B_R} \|w_R(j) - \hat{x}_R(j) \| \right\} + \max \left\{ \frac{\eta^3}{2^5} , \frac{\eta^6}{2^{11}} \right\} \\
&<\frac{\eps}{2}+\frac{\eta^3}{2^5} < \eps.
\end{align*}
Secondly, we have that
\begin{align*}
\|S-T\| &\leq \|SP_{A_L,A_R}-TP_{A_L,A_R}\|+\|TP_{A_L,A_R}-T\|\\
&= \left\|\frac{\widetilde{S}}{\|\widetilde{S}||}-\hat{T}\right\|+\|TP_{A_L,A_R}-T\|\\
&\leq \left\|\frac{\widetilde{S}}{\|\widetilde{S}||}- \widetilde{S}\right\| + \|\widetilde{S} - Q\| + \|Q - R\| + \|R - \hat{T} \| + \|TP_{A_L,A_R}-T\| \\
&< \left| 1 - \|\widetilde{S}\| \right| + \|Q\left(P_{B_L,B_R}+ WP_{A_L\setminus B_L,A_R\setminus B_R}\right)-Q\|  + \frac{\eta^3}{2^4} + \eta + \frac{\eps}{2^4} \\
&< \left|1-\|R\|\right| + \|\hat{T}\left(P_{B_L,B_R}+ WP_{A_L\setminus B_L,A_R\setminus B_R}\right)-\hat{T}\|  +2\|Q-\hat{T}\|+ \frac{\eta^3}{2^4} + \eta + \frac{\eps}{2^4} \\
&\leq \frac{\eps}{2^4}+\eta +  \|\hat{T}P_{B_L,B_R}-\hat{T}\| +\| \hat{T}P_{A_L\setminus B_L,A_R}\|+\|\hat{T}P_{A_L,B_R\setminus A_R}\| + 3 \left( \frac{\eta^3}{2^4} +\eta \right) + \frac{\eps}{2^4} \\
&< \frac{\eps}{2^4}+\eta+ \frac{\eps}{2^2} +  3 \left( \frac{\eta^3}{2^4} +\eta \right) + \frac{\eps}{2^4} \\
&< \eps.
\end{align*}
\end{proof}

Finally, we give our last result which is about the $\mathbf{L}_{p,p}$ for bilinear forms. As we derived Theorem \ref{theorem:c_0-L-p,p} from Theorem \ref{theorem:c_0-uniformly-convex-BPBp}, we deduce the following result on the pair $(c_0(X),c_0(X))$ when $X$ is micro-transitive from Theorem \ref{theorem:c_0-bilinear-form-BPBp}. 

\begin{theorem}
Let $X$ be a micro-transitive Banach space. Then, the pair $(c_0(X),c_0(X))$ has the $\mathbf{L}_{p,p}$ for bilinear forms. 
\end{theorem}

\begin{proof}
Since the proof is very similar with the former one, we just present a sketch of proof instead of giving the full details.

For a given point $(x_L,x_R) \in S_{c_0(X)} \times S_{c_0(X)}$ and $\eps>0$, we take $\gamma(\eps) := \min \left\{ m_{x_L},m_{x_R}, \eta\left(\frac{\theta(\frac{\eps}{3})}{2}\right), \frac{\delta_{\mathbb{C}}\left(\frac{\eps}{6}\right)}{1+\delta_{\mathbb{C}}\left(\frac{\eps}{6}\right)} \right\}$ where $m_{x} := \min_{i \in \N} \{ 1- \|x(i)\| \colon \|x(i)\|<1 \}$, $\delta_{\mathbb{C}}$ is the modulus of $\C$-convexity of $\ell_1(X^*)$, $\eta(\cdot)$ is the function in the definition of the BPBp for bilinear forms of the pair $(c_0(X),c_0(X))$ and $\theta(\cdot)$ is the funtion in the definition of micro-transitivity of $X$. Note that according to the proof of  Theorem \ref{theorem:c_0-bilinear-form-BPBp}, the pair $(\ell_\infty^{C_L}(X),\ell_\infty^{C_R}(X))$ has the BPBp for bilinear forms with the function $\eta(\cdot)$ for arbitrary finite sets $C_A,C_B\subset \mathbb{N}$. 

 Let $T \in \Lin^2(c_0(X) \times c_0(X))$ with $\|T\|=1$ satisfy
$$
|T(x_L,x_R)| > 1-\frac{\gamma^{4}}{2^{6}}.
$$
By taking a suitable rotation, we assume that 
$$
\re T(x_L,x_R) > 1-\frac{\gamma^{4}}{2^{6}}.
$$
As we did in the proof of Theorem \ref{theorem:c_0-bilinear-form-BPBp}, we take suitable subsets $A_L$ and $A_R$ of $\N$ so that 
$$\re TP_{A_L,A_R}(x_L,x_R)=\re T(P_{A_L}x_L,P_{A_R}x_R) > 1-\gamma \quad \text{and} \quad \|TP_{A_L,A_R}-T\|<\frac{\eps}{6}.$$

For the restriction $\hat{T} \in S_{\Lin^2(\ell_\infty^{A_L}(X) \times \ell_\infty^{A_R}(X))}$ of $\frac{TP_{A_L,A_R}}{\|TP_{A_L,A_R}\|}$ and $(\hat{x}_L,\hat{x}_R)=((x_L(i))_{i\in A_L},(x_R(i))_{i\in A_R})\in S_{\ell_\infty^{A_L}(X)\times \ell_\infty^{A_R}(X)}$, apply the BPBp for bilinear forms on the pair $(\ell_\infty^{A_L}(X),\ell_\infty^{A_R}(X))$ to obtain a norm attaining bilinear form $\hat{S} \in S_{\Lin^2(\ell_\infty^{A_L}(X) \times \ell_\infty^{A_R}(X))}$ and its norm attaining point $(\hat{z}_L,\hat{z}_R) \in S_{\ell_\infty^{A_L}(X)}\times S_{\ell_\infty^{A_R}(X)}$ such that 
$$\max\left\{\left\|\hat{S}-\hat{T}\right\|,\left\|(\hat{z}_L,\hat{z}_R)-(\hat{x}_L,\hat{x}_R)\right\|\right\}<\frac{\theta(\frac{\eps}{3})}{2}.$$

Note that the convexity of norm gives that $\left(\left(\frac{\hat{z}_L(i)}{\|\hat{z}_L(i)\|} \right)_{i\in A_L},\left(\frac{\hat{z}_R(i)}{\|\hat{z}_R(i)\|} \right)_{i\in A_R}\right)$ is a norm attaining point of $\hat{S}$. From the micro-transitivity of $X$, there exist isometries $U_i,V_j : X \to X$ for each $i\in A_L$ and $j\in A_R$ so that $U_ix_L(i)=\frac{\hat{z}_L(i)}{\|\hat{z}_L(i)\|}$,  $V_ix_R(j)=\frac{\hat{z}_R(j)}{\|\hat{z}_R(j)\|}$ and $\max_{i\in A_L, j\in A_R}\{\|U_i-I\|,\|V_j-I\|\}<\frac{\eps}{3}$ since $x_L(i)$ and $x_R(j)$ are unit vectors by the definition of $A_L$ and $A_R$ and moreover that
$$\max_{i\in A_L, j\in A_R}\left\{\left\|\frac{\hat{z}_L(i)}{\|\hat{z}_L(i)\|}-x_L(i)\right\|,\left\|\frac{\hat{z}_R(j)}{\|\hat{z}_R(j)\|}-x_R(j)\right\|\right\}<\theta\left(\frac{\eps}{3}\right).
$$
Then, the the canonical extension $S\in \Lin^2(c_0(X) \times c_0(X)) $ of $\widetilde{S}\in \Lin^2(\ell_\infty^{A_L}(X) \times \ell_\infty^{A_R}(X))$ where $\widetilde{S}$ is defined by 
$$\widetilde{S}(y_L,y_R):=\hat{S}\left(\left(U_iy_L(i)\right)_{i\in A_L},\left(V_iy_R(i)\right)_{i\in A_R}\right) \quad \text{for } (y_L,y_R)\in\ell_\infty^{A_L}(X) \times \ell_\infty^{A_R}(X)$$
is the desired bilinear form which means $|S(x_L,x_R)|=\|S\|=1$ and $\|S-T\|<\eps$. 
\end{proof}

\end{document}